\documentclass[10pt]{article}
\usepackage{graphicx}
\usepackage{latexsym}
\usepackage{amsmath}
\usepackage{amssymb}
\usepackage{amsthm}
\usepackage{hyperref}
\usepackage{tikz}

\def \N {{\mathbb N}}

\def \R {{\mathbb R}}

\def \e {{\mathbf e}}

\newtheorem{theorem}{Theorem}[section]

\newtheorem{cor}[theorem]{Corollary}
\newtheorem{lemma}[theorem]{Lemma}
\newtheorem{ex}[theorem]{Example}
\newtheorem{obs}[theorem]{Observation}
\newtheorem{note}[theorem]{Note}

\title{Binary strings of length $n$ with $x$ zeros and longest $k$-runs of 
zeros}

\author {Monimala Nej, A. Satyanarayana Reddy\\
Department of 
Mathematics, Shiv Nadar 
University, India-201314\\ (e-mail: 
monimalanej@gmail.com, satyanarayana.reddy@snu.edu.in).
  }

\date{}
\begin{document}
\maketitle
\begin{abstract}
In this paper, we study $F_{n}(x,k)$, the number of binary strings of length $n$ 
containing $x$ 
zeros and a longest subword of $k$ zeros. A recurrence 
relation for $F_{n}(x,k)$ is derived. We expressed few known numbers like 
Fibonacci, triangular, number of binary strings of length $n$ without $r$-runs 
of ones and number of compositions of $n+1$ with largest summand $k+1$ in terms of $F_{n}(x,k).$ Similar results and applications were obtained  for \^F$_{n}(x,k),$ the number of all palindromic binary strings of length $n$ containing $x$ zeros and longest $k$-runs of zeros.
\end{abstract}
{\bf{AMS(2010)}}: 05A10, 05A15, 05A19,11B39

\section{Introduction} Let $B_n$ denote the set of all binary strings of 
length $n$ and  $B_n^{x,k}$ denote the 
set of all binary strings of length $n$ with $x$ zeros and having at least 
one  longest subword of zeros of length $k.$ For example, 
$$B_6^{4,2}=\{100100,010100,010010,001100,001010,001001\}.$$ 
Consequently, a  necessary condition for  $B_n^{x,k}$ to be   nonempty is $n\ge 
x\ge 
k\ge 0.$  An immediate observation is that $B_n=\cup_{x=0}^n\cup_{k=0}^x 
B_n^{x,k}. $ If $|S|$ denotes the cardinality of the set $S$, then we have 
\begin{equation}\label{eq:Bn=Bnxk}
 2^n=\sum\limits_{x=0}^n\sum\limits_{k=0}^x F_n(x,k),
\end{equation}
where $F_n(x,k)=|B_n^{x,k}|.$ The value of $F_n(x,k)$ is defined to be zero 
whenever $n< 0.$ Many counting problems on binary strings can 
be expressed in terms of $F_n(x,k).$ For example,
\begin{enumerate}
 \item  the number of binary strings of length 
$n$ with $x$ zeros is $\binom{n}{x}.$ Hence, we have 
\begin{equation}\label{eq:ncx}
 \binom{n}{x}=\sum\limits_{k=0}^xF_n(x,k).
\end{equation}
\item the number of binary strings of length $n$ with at least $r$ consecutive 
zeros 
is equal to $\sum\limits_{x=r}^n\sum\limits_{k=r}^x F_n(x,k).$
\item  the number of binary strings  
of  length $n$ with no consecutive zeros is  $f_n$ (for  
example, see \cite{KR,GR1}),  where 
$f_n=f_{n-1}+f_{n-2}$, $n\ge 3$ and  $f_1=2, f_2=3.$ Hence,
\begin{equation}\label{eq:fn=Fnxk}
 f_n=\sum_{x=0}^n 
F_n(x,0)+ \sum_{x=0}^n F_n(x,1)=1+\sum_{x=1}^n F_n(x,1).
\end{equation}
\item In particular,  in our work~\cite{Moni:Satya} the number of primitive symmetric companion matrices with a given exponent is expressed in terms of 
$F_n(x,k).$
\end{enumerate}
 Let $C_n$ denote the $n^{th}$ Catalan number. It is known   
that  $C_n= \binom{2n+1}{n+1}-2\binom{2n}{n+1}$ (see~\cite{Koshy}), hence, from 
Equation~(\ref{eq:ncx}) it follows that every Catalan number can be written in  
terms of $F_n(x,k).$ In Equation~(\ref{eq:fn=Fnxk}) we showed that Fibonacci 
numbers 
can be  expressed in terms of $F_n(x,k)$, similarly, we will show that 
triangular numbers, oblong numbers, tetrahedral numbers  are also expressed in 
terms of $F_n(x,k)$. 

In Section~\ref{sec:FNXK>0}, we will address the problem of  finding all 
$(x,k)$ such that $B_n^{x,k}\ne \emptyset.$ We will also count the number of such 
sets.
In Section~\ref{sec:Formula} we will give a recurrence  formula for $F_n(x,k).$ 
In the last Section~\ref{sec:Applications}, we will show that the problems 
posed 
in \cite{NY1} can  be expressed in terms of $F_n(x,k).$ One to one correspondence between the
set of compositions of $n+1$ with largest summand $k+1$ and 
$B_n^{x,k}$ is provided in Subsection~\ref{sec:Partitions}. Similar results were obtained for palindromic binary words in Section~\ref{section:PBW}.

\section{Finding all $(x,k)$ with $F_n(x,k)>0.$}\label{sec:FNXK>0}
The condition that  $n\ge x\ge k\ge 0$ for $F_n(x,k)>0$  is necessary but not 
sufficient. For example, for $n=5$,  $B_5^{x,k}\ne \emptyset$ if and only if 
$(x,k)\in \{(0,0), (1,1),(2,2),(3,3),(4,4),(5,5), 
(2,1),(3,1),(3,2),(4,2),(4,3)\}.$ Also note that if $k=0$, then $x=0$ and 
$\underbrace{11\cdots111}_{\text{$n$ times}}$ is the only element in  
$B_n^{0,0}.$ The following result 
establishes  upper bound on $x$  (or equivalently lower bound for $n-x$) for a 
given $n,k\in \N$ such that $F_n(x,k)>0.$ For example, if $n=5$ and $k=1$, then 
 the maximum value $x$ can  assume  is $3.$  
Recall if $y\in \R$, then  $\lfloor y \rfloor$ denotes the largest integer 
less than or equal $y$.
 
\begin{lemma}\label{lem:boundonx}
Let $n,k\in \N$ and $n\ge k$. Then  $$F_{n}(x,k) > 
0\Leftrightarrow \begin{cases}
                x+ \lfloor \frac{x}{k} \rfloor \leq n  & \mbox{if $k\nmid x$,}\\

x+ \frac{x}{k} -1  \leq n & \mbox{if $k | x$.}
               \end{cases}$$
\end{lemma}
\begin{proof} Suppose that $F_n(x,k)>0$ or equivalently $X\in  B_n^{x,k}$. Then, 
by definition, $X$ contains $x$ $0's$ and $n-x$ $1's$, further $X$ does not 
contain a subword with only $0's$ of length $k+1$ or more. Hence, the maximum 
number of subwords in $X$  with only $0's$ and of length $k$ are at most 
$\lfloor 
\frac{x}{k} \rfloor$ or $ \frac{x}{k}$ depending on $k\nmid x$ 
or $k|x$ respectively. Or equivalently, 
$\lfloor \frac{x}{k} \rfloor \leq n-x$ if $k\nmid x$, and $ \frac{x}{k} 
-1 \leq n-x$   if $k | x.$

Conversely, suppose that  $\lfloor \frac{x}{k} \rfloor \leq n-x$ if $k\nmid x$, 
and $ \frac{x}{k} 
 -1 \leq n-x$   if $k | x.$ We have to show that $B_n^{x,k}\ne \emptyset$ 
 in each case.
\begin{description}
 \item[$k\nmid x$:] Suppose $\lfloor \frac{x}{k} \rfloor \leq n-x$,  
 then the 
string  $X=\underbrace{YY\cdots YY}_{\mbox{$\lfloor \frac{x}{k} \rfloor$ 
times}}\;\;\;\underbrace{00\cdots00}_{\text{$x-k\lfloor \frac{x}{k} 
\rfloor$ times }}\;\;\;\underbrace{11\cdots1}_{\text {$n-x- \lfloor 
\frac{x}{k} \rfloor$ 
 times}}\in B_n^{x,k},$ where $Y=\underbrace{000\cdots 0}_{\text{$k$ times}}1$ 
be a subword of $X$ of length $k+1$ containing $k$ consecutive $0's$  
followed by a $1$.
\item[$k|x$:] Suppose  $ \frac{x}{k}-1 \leq n-x$. Then   the string 
$X=\underbrace{YY\cdots YY}_{\mbox{$\frac{x}{k}-1$ 
times}}\:\;\underbrace{000\cdots 
0}_{\text{$k$ times}}\;\;\underbrace{11\cdots1}_{\text {$n-x- 
\frac{x}{k}-1$ 
 times}}\in B_n^{x,k}.$
\end{description}
\end{proof}
In the above lemma, we found an upper bound for  $x$ for a given $n$ and $k$. 
Since $F_n(x,x)>0$ for all $0\le x\le n$, hence, for a given $n$ and $x$ the 
maximum value for $k$ such that $F_n(x,k)>0$ is $x$. For  fixed  $n$ and $x$, 
the following result will provide the least  value of $k$ such that 
$F_n(x,k)>0$. 
\begin{cor}\label{cor:boundonk}
 Let $n,x\in \N$ be fixed and $F_n(x,k)>0$. Then the smallest  value of  $k$ is
$\left\lfloor 
\frac{n}{n-x+1} \right\rfloor.$
\end{cor}

\begin{theorem}\label{thm:|S_n|}
Let $n \in \mathbb{N}$ and $S_{n}=\{(x,k):F_{n}(x,k)>0\}$. Then 
$|S_{n}|=\binom{n+2}{2}-\sum\limits_{i=0}^{n} \lfloor \frac{n}{i+1} 
\rfloor$.
\end{theorem}
\begin{proof}
For each $i$, $0\le i \le n$, we define the sets 
$$B_{i}=\left\{0,1,2,\ldots,(n-i)-\left\lfloor \frac{n}{i+1} 
\right\rfloor\right\}\;\; and 
\;\;A_{i}=\left\{(n-i,\left\lfloor \frac{n}{i+1} \right\rfloor +j):j \in 
B_{i}\right\}.$$
Then from Corollary~\ref{cor:boundonk}, we have 
$S_{n}=\bigcup\limits_{i=0}^n 
A_{i}$.

Hence,   the result is derived  
from the fact that 
$$|S_n|=\sum\limits_{i=0}^n|A_i|=\sum\limits_{i=0}^n|B_i|=\sum\limits_{i=0}
^n (n-i+1)-\left\lfloor \frac{n}{i+1} 
\right\rfloor=\binom{n+2}{2}-\sum\limits_{i=0}^{n} 
\left\lfloor\frac{n}{i+1}\right\rfloor.$$
\end{proof}
\section{Formula for $F_n(x,k)$}\label{sec:Formula}
In this section we will  find a recurrence formula for $F_n(x,k)$.  First we 
explore the values of $F_n(x,k)$,   which will follow immediately from its 
definition. 

\begin{enumerate}
 \item Let $1\le x\le n.$, then  $F_{n}(x,x)=(n-x)+1.$  Consequently,  
we have $\{F_n(1,1)|n\in \N\}=\N.$
\item From  Equation~(\ref{eq:ncx}) and $F_n(2,2)$ we have  
$F_{n}(2,1)=\frac{(n-1)(n-2)}{2}$ for all $n\ge 2.$\\
 Hence, the set $\{F_{n}(2,1):n\in\mathbb{N},\;n\geq2\}$ is the set of 
triangular numbers.
 \item Let $x,n\in \N\setminus\{1,2\}.$, then  $F_{n}(x,x-1)=(n-x)(n-x+1).$ 

\begin{proof}
Suppose $a_1a_2\dots a_n$ be a binary string  of length $n$ which 
contains $x$ zeros among which $x-1$  are consecutive. Then 
$a_i=a_{i+1}=a_{i+2}=\dots=a_{i+x-2}=0$, for some $i 
\in\{1,2,\ldots,n-x+2\}$, further $a_{i-1}=a_{i+x-1}=1$, whenever $i\ne 1, 
n-x+2$. If  $i=1$, then $a_{x}=1$ 
and if $i=n-x+2$, then $a_{n-x+1}=1$. Thus, for each $i 
\in\{2,3,\ldots,n-x+1\}$, there are 
$n-x-1$ possible choices to place remaining zero. And for 
$i=1 \;or \; i=n-x+1$, 
 there are $(n-x)$  possible choices to place remaining zero. Hence, 
$F_{n}(x,x-1)=(n-x-1)(n-x)+2(n-x)=(n-x)(n-x+1)$.
\end{proof}
Thus, for $x\ge 3,$ the sequence $\{F_n(x,x-1)\}$ is the sequence 
of 
oblong numbers. 
Similarly one can prove that the sequence $\{F_n(3,1)|n\ge 5\}$ is the sequence of 
tetrahedral numbers.
\end{enumerate}
The following result computes $F_{n}(x,k)$ explicitly whenever $x-k<k.$
\begin{theorem}
Let $n-1>x\ge k\ge 1$ such that $F_{n}(x,k)>0$ and $x<2k$, then $$F_{n}(x,k) = 
2\binom{n-k-1}{x-k}+(n-k-1)\binom{n-k-2}{x-k}.$$ 
\end{theorem}                  
\begin{proof}
Let $a_1a_2\dots a_n\in B_n^{x,k}$. Suppose $a_{1}=a_{2}=\dots=a_{k}=0$, 
then we have $a_{k+1}=1$. Since $x-k<k$, there are  
$\binom{n-k-1}{x-k}$ different strings in $B_n^{x,k}$ with first 
$k$ 
bits as zero. Similarly, there are  $\binom{n-k-1}{x-k}$ different strings in 
$B_n^{x,k}$ with the  last $k$ bits equal to 
zero.  If we 
choose $k$ consecutive zeros as $a_{i}=a_{i+1}=\dots=a_{i+k-1}=0$, for  $i 
\in \{2,3,\ldots,n-k\}$, then $(x-k)$ zeros can be placed at any of the 
remaining $n-(k+2)$ places in the string. Thus, for each such 
$i$, there are  $\binom{n-k-2}{x-k}$ strings which belong to $B_n^{x,k}.$
Hence, the result follows.
\end{proof}
Using above theorem, we can compute $F_n(n-1,k)$, for all possible values of $k$ 
except when $n$ is 
odd and $k=\frac{n-1}{2}$. In this case,  the value of
$F_n(n-1,k)$ will be $1.$

The following result provides a recurrence relation for  $F_n(x,k).$ The 
results which we proved until now in this section will supply the necessary initial 
conditions. Thus, we can evaluate $F_n(x,k)$ for any $n,x,k.$
\begin{theorem}\label{thm:main:recurrence}
Let $n\in \N\setminus \{1,2\}$, $1 \leq 
x \leq n-2$ and $\lfloor \frac{n}{n-x+1} \rfloor \leq k \leq x$. Then 
$$F_{n}(x,k)=\sum\limits_{i=0}^{k-1}F_{n-i-1}(x-i,k)+\sum\limits_{j=0}^{k}F_{
n-k-1}(x-k,j).$$ 
\end{theorem}                  
\begin{proof}
Let us denote $G_n^{x,k}$ and $H_n^{x,k}$ as set of all strings in $B_n^{x,k}$ 
starting with $0$ and $1$ respectively. Then  clearly $B_n^{x,k}=G_n^{x,k} \cup 
H_n^{x,k}$ and hence, $F_{n}(x,k)=|B_n^{x,k}|=|G_n^{x,k}| + |H_n^{x,k}|.$ Thus, 
it is sufficient to find $|G_n^{x,k}|$ and $|H_n^{x,k}|.$

If  $y \in H_n^{x,k}$, then $y$ is of the form $1y_1,$ where $y_1 \in 
H_{n-1}^{x,k} \cup G_{n-1}^{x,k}.$ Hence, 
\begin{equation}\label{eq:hnxk}
|H_n^{x,k}|=|H_{n-1}^{x,k}|+|G_{n-1}^{x,k}|=|B_{n-1}^{x,k}|. 
\end{equation}

Now suppose  $y \in G_n^{x,k}$, then  $y$ is of the form 
$\underbrace{00\dots 0}_{\text{$i$ times}}y_{1}$ or 
$\underbrace{00\dots 0}_{\text{$k$ times}}y_{2}$  where $y_{1} \in 
H_{n-i}^{x-i,k}$, 
$y_{2} \in H_{n-k}^{x-k,j}$, $1 \leq i \leq k-1$ and $0 \leq j \leq k$ or 
equivalently 
$|G_n^{x,k}|=\sum\limits_{i=1}^{k-1}|H_{n-i}^{x-i,k}|+\sum\limits_{
j=0}^{k}|H_{n-k}^{x-k,j}|.$ Hence, the result follows from 
Equation~(\ref{eq:hnxk}).
\end{proof}                   

For example, 
\begin{eqnarray*}
 F_{6}(3,2)&= & \sum\limits_{i=0}^{1}F_{6-i-1}(3-i,k)+\sum\limits_{j=0}^{2}F_{
6-2-1 }
(3-2,j)\\        
&=&F_{5}(3,2)+F_{4}(2,2)+F_{3}(1,0)+F_{3}(1,1)+F_{3}(1,2)=6+3+0+3+0=12.
\end{eqnarray*}
 And $\{001011, 001101, 010011, 011001,100110,101100,
 101001,110010,001110, 011100,100101,110100\}$ is the set $B_6^{3,2}.$
 Now we will give a direct proof for Equation~(\ref{eq:fn=Fnxk}). That is 
$1+\sum_{x=1}^n F_n(x,1)$ satisfies the recurrence relation  
$f_n=f_{n-1}+f_{n-2}$, $n\ge 3$ with initial conditions  $f_1=2, f_2=3.$  Let 
$a_n=1+\sum_{x=1}^n F_n(x,1).$  Then $a_1=1+F_1(1,1)=2$ and 
$a_2=1+F_2(1,1)+F_2(2,1)=3.$ We will prove result for $n>2$ using 
Theorem~\ref{thm:main:recurrence}. 
\begin{eqnarray*}
 a_n  &=& 1+\sum_{x=1}^n 
F_n(x,1)= 1+\sum_{x=1}^n [F_{n-1}(x,1)+F_{n-2}(x-1,0)+F_{n-2}(x-1,1)]\\
&=& 
\left(1+\sum_{x=1}^{n-1}F_{n-1}(x,1)\right)+\left(1+\sum_{x=1}^{n-2}F_{n-2}(x-
1, 1)\right)=a_{n-1}+a_{n-2}.
\end{eqnarray*}

We now express a few of the above identities in terms of matrix equations by 
constructing a matrix using the numbers $F_n(x,k).$ 
For a given $n$, $F_n$ denotes an $(n+1) \times (n+1)$ matrix and is defined as 
$F_n=[a_{xk}]$,
where $a_{xk}=F_n(x,k)$ for $0\le x,k\le n.$
For example, $$F_1=\begin{bmatrix}
                   1 & 0\\
                   0 &1
                  \end{bmatrix}, F_2=\begin{bmatrix}
                  1& 0& 0,\\
                  0& 2& 0\\
                  0& 0& 1
                  \end{bmatrix}, F_3=\begin{bmatrix}
                  1 & 0& 0& 0\\
                  0& 3& 0& 0\\
                 0& 1& 2& 0\\
       0& 0& 0& 1
                  \end{bmatrix}, F_4=\begin{bmatrix}
                  1& 0& 0& 0& 0\\
         0& 4& 0& 0& 0\\
        0& 3& 3& 0& 0\\
        0& 0& 2& 2& 0\\
        0& 0& 0& 0& 1\\
        \end{bmatrix}.$$
Let $\e_n$ denote the column vector in $\R^n$ with each entry as $1.$  Then the following hold.
\begin{enumerate}
 \item From the Equation (\ref{eq:Bn=Bnxk}), we have $\e_{n+1}^TF_n\e_{n+1}=2^n.$
 \item From the Equation (\ref{eq:ncx}), we have 
 $$F_n \e_{n+1}=\left[\binom{n}{0}, \binom{n}{1}, \ldots, 
\binom{n}{n}\right]^T,$$
 the $n^{th}$ row of Pascal triangle. 
 \item The row vector $\e_{n+1}^TF_n$ provides column sums of $F_n.$ 
 Since the first column of $F_n$ is the vector $[1,0,\ldots,0]^T.$ Hence, the 
sequence
 of first column sum of $F_n$ for $n=1,2,\dots$ is  
 $((\e_{n+1}^TF_n)_1)_{n=1}^\infty=(1,1,1,1,\ldots).$ From 
Equation (\ref{eq:fn=Fnxk}), we have 
 $((\e_{n+1}^TF_n)_2)_{n=1}^\infty=(1,2,4,7,12,20,33,54,88,\ldots).$
   \item From Theorem~\ref{thm:|S_n|}, the number of nonzero entries in the 
matrix $F_n$ is equal to
 $|S_n|.$  
 \end{enumerate}
 The eigenvalues of $F_n$ are $1,1,2,3,\ldots,n.$ Hence, the trace and 
determinant of $F_n$ 
 are respectively  $1+\frac{n(n+1)}{2}$ and $n!.$
  The matrix $F_n$ is diagonalizable as $[1,0,0,\ldots,0]^T$ and 
 $[0, 0,0,\ldots,1]^T$ are eigenvectors corresponding to eigenvalue $1.$ 
 We will explore a few more properties of these matrices in our future work.
 
\section{Applications}\label{sec:Applications}
   Equation (\ref{eq:fn=Fnxk}) counts the number of binary strings  
having no pair of consecutive $0$'s, (or  no pair of consecutive 
$1$'s). A binary string with no subword of length $r$ of zeros for $r\ge 1$
is  called as  a binary string having  no $r$-runs of $0$'s \cite{NY1,NY2,GR2}.
  R. Grimaldi 
and S. Heubach~\cite{GR2}, showed  that the total number of binary strings of 
length $n$ having no odd  runs of 
$1$'s is equal to $f_{n+1}$, the $(n+1)^{th}$ Fibonacci number. M. A. Nyblom 
~\cite{NY1}  related $r$-Fibonacci  sequences for any fixed integer $r\geq 2$ 
to the total number of binary strings of length $n$ without $r$-runs of $1$'s.

 M. A. Nyblom in~\cite{NY1} denoted $S_{r}(n)$ as the set of all binary strings 
of length $n$ without  $r$-runs of ones, where $n\in \mathbb{N}$ and $r \geq 
2$, and $T_{r}(n)=|S_{r}(n)|$. For example, if  $n=3$, $r=2$, then 
$S_{2}(3)=\{000, 101,001,100,010\}$, $T_{2}(3)=5$.  For $n>r,$ M. A. Nyblom \cite{NY1} 
showed that \begin{equation*}
T_{r}(n)=\sum\limits_{i=1}^{r}T_{r}(n-i)
\end{equation*}
with $r$ initial conditions $T_{r}(s)=2^s$ for $s=1,2,\ldots,r-1$ and 
$T_{r}(r)=2^r-1$. The following result 
expresses $T_{r}(n)$ in terms of $F_{n}(x,k).$

 \begin{theorem}\label{thm:Trn=Fnxk}
  Let $n,r\in \N$ and $r\ge 2.$ Then 
$T_{r}(n)=1+\sum\limits_{k=1}^{r-1}\sum\limits_{x= 
k}^{n}F_{n}(x,k).$
 \end{theorem}
 \begin{proof} It is easy to see that initial conditions follow from the 
Equation~(\ref{eq:Bn=Bnxk}).  For $s=1,2,\ldots,r-1$,
\begin{eqnarray*}
 T_{r}(s)&=& 1+\sum\limits_{k=1}^{r-1}\sum\limits_{x=k}^{s} F_{s}(x,k)       
=1+\sum\limits_{k=1}^{s}\sum\limits_{x=k}^{s} 
F_{s}(x,k)= \sum\limits_{x=0}^s\sum\limits_{k=0}^x F_s(x,k)=2^s.
\end{eqnarray*}
 
 For $s=r$, \begin{eqnarray*}
T_{r}(r)&=& 1+\sum\limits_{k=1}^{r-1}\sum\limits_{x=k}^{r} 
F_{r}(x,k)=\left(\sum\limits_{x=0}^r\sum\limits_{k=0}^x 
F_r(x,k)\right)-F_{r}(r,r)=2^r-1.             
            \end{eqnarray*}
From  Theorem~\ref{thm:main:recurrence}, we have 
$F_n(x,k)=\sum\limits_{i=0}^{k-1}F_{n-i-1}(x-i,k)+\sum\limits_{j=0}^{k}F_{n-k-1}
(x-k,j)$.\\ Applying this to $F_n(x,1),F_n(x,2),\ldots, F_n(x,r-1),$ we get
\begin{eqnarray*}
F_n(x,1) &=&  F_{n-1}(x,1)+F_{n-2}(x-1,0)+F_{n-2}(x-1,1)\\
F_n(x,2) &=& F_{n-1}(x,2)+F_{n-2}(x-1,2)+F_{n-3}(x-2,0)+F_{n-3}(x-2,
1)+F_{n-3}(x-2,2)\\\nonumber
\vdots\\\nonumber
F_n(x,r-1) &=&F_{n-1}(x,r-1)+F_{n-2}(x-1,r-1)+\dots + 
F_{n-r+1}(x-r+2,r-1)\\&&+F_{n-r}(x-r+1,0)+F_{n-r}(x-r+1,1)+\dots 
+F_{n-r}(x-r+1,r-1).
\end{eqnarray*}
In order to show $T_r(n)=1+\sum\limits_{k=1}^{r-1}\sum\limits_{x= 
k}^{n}F_{n}(x,k)$ for $n>r$, it is sufficient to show that  
 $1+\sum\limits_{k=1}^{r-1}\sum\limits_{x= 
k}^{n}F_{n}(x,k)$ satisfies the recurrence relation  
$T_{r}(n)=\sum\limits_{i=1}^{r}T_{r}(n-i).$
\begin{eqnarray}\label{eq:thm4.1}
 1+\sum\limits_{k=1}^{r-1}\sum\limits_{x=k}^{n}F_n(x,k) 
&=&1+\sum\limits_{x=1}^{n}F_{n}(x,1)+\sum\limits_{x=2}^{n}F_{n}(x,
2)+\dots+\sum\limits_{x=r-1}^{n}F_{n}(x,r-1)
\end{eqnarray}
Now by substituting the values of  $F_n(x,1),F_n(x,2),\ldots, F_n(x,r-1)$, and 
by using the fact that 
$\sum\limits_{x=k}^{n}F_{n-i}(x-t,k)=\sum\limits_{x=k}^{n-i} 
F_{n-i}(x,k)\;\;for\; k \geq t \;\;and \;i \geq k$,  we 
can rewrite Equation~(\ref{eq:thm4.1}) as
\begin{eqnarray*}
 1+\sum\limits_{k=1}^{r-1}\sum\limits_{x=k}^{n}F_n(x,k) 
&=&\left(1+\sum\limits_{k=1}^{r-1}\sum\limits_{x=k}^{n-1}
F_{n-1}(x,k)\right)+\left(1+\sum\limits_{k=1}^{r-1}\sum\limits_{x=k}^{n-2}F_{n-2
}(x,k) \right)
+\dots+\\
&&\left(1+\sum\limits_{k=1}^{r-1}\sum\limits_{x=k}^{n-r}F_{n-r}(x,
k)\right).
\end{eqnarray*}
Hence, the result follows.
 \end{proof}
 Let  $O_{n}(x,k)$ denote  the total number 
of ones in  $B_{n}^{x,k}$ and $O_{n}(r)$ as the total number of zeros in 
$S_{r}(n)$, for example $O_{2}(3)=10$. It is easy to see that 
$O_{n}(x,k)=(n-x)F_{n}(x,k).$ For $n > r,$ M. A. Nyblom in~\cite{NY1} 
proved  that \begin{equation*}
O_{r}(n)=\sum\limits_{i=1}^{r}O_{r}(n-i)+T_{r}(n) \nonumber
\end{equation*}
with $r$ initial conditions $O_{r}(s)=s2^{s-1}$ for $s=1,2,\ldots,r.$

\begin{theorem}\label{thm:ONR=ONXK}
 Let $n,r\in 
\N$ and $r\ge 2.$ Then 
$O_{r}(n)=\sum\limits_{k=0}^{r-1}\sum\limits_{x=k}^{n}O_{n}(x,k).$
\end{theorem}
\begin{proof}
In order to show 
$O_{r}(n)=\sum\limits_{k=0}^{r-1}\sum\limits_{x=k}^{n}O_{n}(x,k),$ we will show for $n > r,$
$\sum\limits_{k=0}^{r-1}\sum\limits_{x=k}^{n}O_{n}(x,k)$ satisfies 
\begin{equation*}
O_{r}(n)=\sum\limits_{i=1}^{r}O_{r}(n-i)+T_{r}(n)  \nonumber
\end{equation*}
with $r$ initial conditions $O_{r}(s)=s2^{s-1}$ for $s=1,2,\ldots,r.$

It is easy to see that initial conditions follows from the 
Equation~(\ref{eq:ncx}).
For $s=1,2,\ldots,r$,  we have
 \begin{eqnarray*}
  O_{r}(s)&= & 0.F_{s}(0,0)+1.\sum\limits_{k=0}^{1} 
F_{s}(1,k)+2.\sum\limits_{k=0}^{2} F_{s}(2,k)+3.\sum\limits_{k=0}^{3} 
F_{s}(3,k)+\dots + s.\sum\limits_{k=0}^{s} F_{s}(s,k)\\
&=& 1.\binom{s}{1}+2.\binom{s}{2}+3.\binom{s}{3}+ \dots+ 
s.\binom{s}{s}=s.2^{s-1}.  
 \end{eqnarray*}

Now we use $O_{n}(x,k)=(n-x)F_{n}(x,k)$ and Theorem~\ref{thm:main:recurrence} 
to 
prove that $\sum\limits_{k=0}^{r-1}\sum\limits_{x=k}^{n}O_{n}(x,k)$ satisfies 
the recurrence 
relation $O_{r}(n)=\sum\limits_{i=1}^{r}O_{r}(n-i)+T_{r}(n)$ for $n > 
r$ which is similar to the  proof of Theorem~\ref{thm:Trn=Fnxk}.

\begin{eqnarray*}
\sum\limits_{k=0}^{r-1}\sum\limits_{x=k}^{n}O_{n}(x,k) &=& 
\sum\limits_{k=0}^{r-1}\sum\limits_{x=k}^{n}(n-x)F_n(x,k)\\ 
&=&\sum\limits_{k=0}^{r-1}\sum\limits_{x=k}^{n-1}(n-x)  
F_{n-1}(x,k) +\sum\limits_{k=0}^{r-1}\sum\limits_{x=k}^{n-2}(n-x-1) 
F_{n-2}(x,k))\\&& +\sum\limits_{k=0}^{r-1}\sum\limits_{x=k}^{n-3}
                                                       (n-x-2)F_{n-3}(x,k)) 
+\dots+\sum\limits_{k=1}^{r-1}\sum\limits_{x=k}^{n-r}                     
                                (n-x-(r-1)) F_{n-r}(x,k))\\        
&=&\sum\limits_{k=0}^{r-1}\sum\limits_{x=k}^{n-1}
                                                   (n-x-1)  
F_{n-1}(x,k)+\sum\limits_{k=0}^{r-1}\sum\limits_{x=k}^{n-2}
                                                      (n-x-2) F_{n-2}(x,k) 
+\dots+\\
&&\sum\limits_{k=1}^{r-1}\sum\limits_{x=k}^{n-r}
                                                      (n-x-r) F_{n-r}(x,k) 
+\sum\limits_{k=0}^{r-1}\sum\limits_{x=k}^{n-1}F_{n-1}
(x,k)+\sum\limits_{k=0}^{r-1}\sum\limits_{x=k}^{n-2}
                                                      F_{n-2}(x,k) 
+\\
&&\sum\limits_{k=0}^{r-1}\sum\limits_{x=k}^{n-3}
                                                       F_{n-3}(x,k) 
+\dots+\sum\limits_{k=1}^{r-1}\sum\limits_{x=k}^{n-r}
                                                       F_{n-r}(x,k) \\
&=& \sum\limits_{i=1}^{r}O_{r}(n-i)+T_{r}(n).
\end{eqnarray*}
Hence, the result follows.
\end{proof}
\subsection{Number of compositions and partitions  of $n+1$}\label{sec:Partitions}
A composition of a positive integer $n$ is a sequence $k_1,k_2,\ldots, k_r$ of 
positive 
integers called 
summands of the composition such that $n=k_1+k_2+\dots+k_r.$ It is well known 
 that 
the number of compositions of $n$ is $2^{n-1}$(see~\cite{Sills},\cite{AT}). There is a one to one 
correspondence between 
compositions of $n$ and binary sequences of length $n-1$ (refer 
Section~5 in ~\cite{GR2}).

Few interesting applications of column sums of $F_n$ are:

\begin{enumerate}
 \item the $k^{th}$ column 
sum of $F_n$ is the 
 number of compositions  of $n+1$  with maximum summand $k+1$. Consequently the 
total number of compositions of $n+1$ is 
$\e_{n+1}^TF_n\e_{n+1}=2^n.$ 
\item  Further, from the Equation~(\ref{eq:fn=Fnxk}),  the number of compositions of 
$n+1$ 
with ones and twos is 
 $(\e_{n+1}^TF_n)_1+(\e_{n+1}^TF_n)_2=1+(\e_{n+1}^TF_n)_2=f_n$, the same is 
derived by ~Krishnaswami Alladi and V.~E.~Hoggatt, Jr in ~\cite{K-H}.
\end{enumerate}
 S.~Heubach et al  ~\cite{H-C-G} counted the number of ``+" signs and number 
summands in all 
 compositions  of $n$. These results can be  expressed in terms of $F_n(x,k)$ 
in the 
following manner.
 \begin{enumerate}
  \item It is clear that the number of ``+" signs in a 
composition of $n$ is the same as the number of $1$'s in the corresponding binary 
string of length $n-1$. Hence, the total number of ``+" signs in all 
compositions 
of $n$  is  
$$\sum\limits_{k=0}^{x}\sum\limits_{x=0}^{n-1} x \; 
F_{n-1}(x,k)=\sum\limits_{x=1}^{n-1}x \; \binom{n-1}{x}=(n-1)2^{n-2}.$$
\item The number of summands in a composition is one more than the number of 
``+'' 
signs, hence, the number of summands in all compositions of $n$ is given by  
$$(n-1)2^{n-2}+\sum\limits_{k=1}^{n-1}\sum\limits_{x=1}^{n-1}F_{n-1}(x,k)+F_{n-1
}(0,0)=(n+1)2^{n-2}.$$
 \end{enumerate}
Let $P_n$ denote the number of partitions of positive integer $n.$ Then it is 
easy to see that 
\begin{equation}\label{eq: P_{n+1}{x,k}}
P_{n+1}=\sum\limits_{(x,k)\in S_{n}}P_{n+1}(x,k),
\end{equation}
where $P_{n+1}(x,k)$ denotes the number of partitions of $n+1$ such that  
 the binary strings corresponding to partitions of $n+1$ belongs to  
$B_{n}^{x,k}$.

The Equation~(\ref{eq: P_{n+1}{x,k}}) can be derived by defining 
a  relation on $B_{n}^{x,k}$ as $X\sim 
Y$ if  the lengths  of subwords of zeros are the
same in both $X$ and $Y.$ Clearly, this is an equivalence relation and  $P_{n+1}(x,k)$ is equal to the
number of equivalence classes of this equivalence relation.
In other words, two elements $X,Y \in B_{n}^{x,k}$ correspond  to distinct partitions of 
$n+1$ if and only if  there exists a positive integer $ l>0 $ such that the number of subwords of zeros of 
length 
$l$ in $X$ and $Y$ are distinct.

\begin{ex}
\begin{enumerate}
\item
For $n=6$, $x=4$, $k=2$ we have $B_{6}^{4,2}=[X]\cup[Y]$
where $X=001001$, $Y=001010$ and
$$[X]=\{001001,001100,100100\}\;\;[Y]=\{001010,010010,010100 \}.$$
Hence, $P_{7}(4,2)=2.$
\item
For $n=10$, $x=7$,  $k=3$ we have $|B_{10}^{7,3}|=36$ and 
$B_{10}^{7,3}=[X]\cup[Y]\cup[Z]$\\
where $X=0001000101$, $Y=0001001001$, $Z=0001001010$ and 
\begin{eqnarray}
[X]=\{0001000101,0001000110,0001100010,0001010001,0001011000,0001101000, 
\nonumber \\
   1000100010,1000101000,0100010001,0100011000,0110001000,1010001000 \} 
\nonumber
\end{eqnarray}
\begin{eqnarray}
[Y]=\{0001001001,0001001100,0001100100,1000100100,0010001001,0010001100, 
\nonumber \\
      0011000100,0010010001,0010011000,0011000100,1001000100,1001001000  \} 
\nonumber
\end{eqnarray}
\begin{eqnarray}
[Z]=\{0001001010,0001010010,0001010100,0010001010,0100010010,
0100010100\nonumber 
\\
      0101000100,0010100010,0100100010,0101001000,0010101000,0100101000 \}. 
\nonumber
\end{eqnarray}
Hence, $P_{11}(7,3)=3.$
\end{enumerate}
\end{ex}

 From the Equation~(\ref{eq: P_{n+1}{x,k}}), we have $|S_n|\le P_{n+1}$, further 
it is easy 
 to check that $|S_n|=P_{n+1}$ holds only when $n\le 5.$ 
 Hence, in order to find $P_{n+1}$  it is sufficient to find $P_{n+1}(x,k)$ for 
all $x,k$ with 
 $0\leq k \leq x \leq n.$  
 
 Before stating the recurrence relation for 
$P_{n+1}(x,k)$, 
 it is easy to see that the following values of $P_{n+1}(x,k)$
 follow immediately from its definition whenever $F_n(x,k)>0$. 
 \begin{enumerate}
  \item  $P_{n+1}(x,x)= P_{n+1}(x,x-1)=P_{n+1}(x,1)=1.$  
 \item If $n = x+\; \lfloor \frac{x+1}{2} 
\rfloor+i-1,$ where $i\ge 0$, then 
 $$P_{n+1}(x,2) = 
 \begin{cases}
             i+1   & \mbox{if $0 \leq i < \lfloor \frac{x}{2} \rfloor -1$,}\\

\lfloor \frac{x}{2} \rfloor & \mbox{if $i \geq \lfloor \frac{x}{2} \rfloor -1 
$.}
               \end{cases}$$
\begin{proof}
Let $X\in B_{n}^{x,2}$ and $n_{00}$ denote the number of times the string
$00$ occurs as a subword  in $X$. Then $ n_{00} \le \lfloor \frac{x}{2} 
\rfloor$ and 
the minimum value of 
$n_{00}$ is either $\lfloor \frac{x}{2} \rfloor-i$ or $1$
depending on   $ i < \lfloor 
\frac{x}{2} \rfloor -1$ or $i \geq \lfloor \frac{x}{2} \rfloor -1$ respectively.
Hence, the result follows.
\end{proof}                
\item If $k|x$ and $n = x+\frac{x}{k}-1$, then $P_{n+1}(x,k) =1.$ 
\item If  $\lfloor \frac{n}{2} \rfloor \le k\le n-1,$ then  $P_{n+1}(n-1,k)=1.$ 
\end{enumerate}
For  $X\in B_n^{x,k},$ we denote $X_l$ as the  multiset of all lengths(with repetitions) of subwords of zeros in $X$.
Then $X\sim Y$ if and only if $X_l=Y_l$. Hence, we can define $[X]_l=X_l.$
If $[X]_l=\{l_1,l_2,\ldots,l_t\},$  then there is a bijection between $[X]$  and the set of all 
compositions of $n+1$  having  summands $l_{1}+1,l_{2}+1,\ldots,l_{t}+1$ and 
$(n-\sum\limits_{i=1}^{t}(l_{i}+1)+1)$ $1$'s. Further for each $i$, $1\le i\le P_{n+1}(x,k)$ there exists
 $Z_i\in [X_i]$ such that $Z_i=\underbrace{00\dots 0}_{\text{$k$ times}}1Y_i$ where 
 $Y_i\in B_{n-k-1}^{x-k,k_i}$ and $0\le \lfloor\frac{n-k-1}{n-x}\rfloor\le k_i\le k.$
 Hence, the following result.
 \begin{theorem}
 Let $n> x\ge k\ge 3.$ Then
 $P_{n+1}(x,k)=\sum\limits_{j=\left\lfloor\frac{n-k-1}{n-x}\right\rfloor}^{k}P_{n-k}(x-k,j).$
 \end{theorem}
\section{Palindromic Binary words}\label{section:PBW}

Let $\hat B_n$ denote the set of all palindromic binary 
strings of 
length $n$ and  $\hat B_n^{x,k}$ denote the 
set of all palindromic binary strings of length $n$ with $x$ zeros and having 
at 
least 
one  largest subword of zeros of length $k.$ For example, 
$$\hat B_6^{4,2}=\{010010,001100\}.$$ 
Consequently, a  necessary condition for  $\hat B_n^{x,k}$ to be   nonempty is 
$n\ge 
x\ge 
k\ge 0.$  An immediate observation is  $\hat B_n=\cup_{x=0}^n\cup_{k=0}^x 
\hat B_n^{x,k}. $ If $|S|$ denotes the cardinality of the set $S$, then we have 
\begin{equation}\label{eq:Bn=Bnxk1}
 2^{\lfloor \frac{n+1}{2}\rfloor}=\sum\limits_{x=0}^n\sum\limits_{k=0}^x \hat 
F_n(x,k),
\end{equation}
where $\hat F_n(x,k)=|\hat B_n^{x,k}|.$ The value of $\hat F_n(x,k)$ is defined 
to be zero 
whenever $n < 0.$ Many counting problems on binary strings can 
be expressed in terms of $\hat F_n(x,k).$ For example,
\begin{enumerate}
 \item  the number of palindromic binary strings of length 
$n$ with $x$ zeros is $\binom{\lfloor \frac{n}{2}\rfloor}{\lfloor 
\frac{x}{2}\rfloor}.$ Hence, we have 
\begin{equation}\label{eq:ncx1}
 \binom{\lfloor \frac{n}{2}\rfloor}{\lfloor 
\frac{x}{2}\rfloor}=\sum\limits_{k=0}^x \hat F_n(x,k).
\end{equation}
\item the number of palindromic binary strings of length $n $ with at least $r$ 
consecutive  zeros is equal to $\sum\limits_{x=r}^n\sum\limits_{k=r}^x \hat 
F_n(x,k).$
\item  the number of palindromic binary strings  
of  length $2n$ and $2n-1$ with no consecutive zeros are respectively
\begin{equation}\label{eq:an=Fnxk}
f'_{n-1}=1+\sum\limits_{i=1}^{n} \hat F_{2n}(2i,1), \; {\text {and}} 
\end{equation}
\begin{equation}
f_n=1+\sum\limits_{x=1}^{2n-1} \hat F_{2n-1}(x,1).\nonumber
\end{equation}
\end{enumerate}
We will see that $f'_{n-1}$ and $f_{n}$ are nothing but $a_{n-1}$ and 
$a_{n}$,  where $a_{n}$ is the $n^{th}$ Fibonacci number, that is
$a_n=a_{n-1}+a_{n-2}$, $n\ge 3$ with  $a_1=2, a_2=3.$ 

In Section~\ref{sec:PFNXK>0}, we will address the problem of  finding all 
$(x,k)$ such that $\hat B_n^{x,k}\ne \emptyset.$ We also count the number of 
such  sets. In Section~\ref{sec:PFormula}, we will give a recurrence  formula for 
$\hat 
F_n(x,k).$  Some results on compositions of $n$ by Alladi and Hoggatt
~\cite{K-H} are  generalized in ~\cite{H-C-G}, in Section~\ref{sec:PApplications} we will show all those results can be 
 expressed in terms of 
$\hat F_{n}(x,k)$.
We will use the same notations  as mentioned  above throughout the paper and 
without (hat)  \;\;  
\^ \;\; to denote the corresponding quantities for binary strings without 
palindromic restriction.

\subsection{Finding all $(x,k)$ with \^F$_n(x,k)>0.$}\label{sec:PFNXK>0}
The condition that  $n\ge x\ge k\ge 0$ for $\hat F_n(x,k)>0$  is necessary but 
not 
sufficient. For example, for $n=5$,  $\hat B_5^{x,k}\ne \emptyset$ if and only 
if 
$$(x,k)\in \{(0,0), (1,1),(3,3),(5,5), 
(2,1),(3,1),(4,2)\}.$$ Also note that if $k=0$, then $x=0$ and 
$\underbrace{11\cdots 111}_{\text{$n$ times}}$ is the only element in  
$\hat B_n^{0,0}.$ The following result 
establishes  upper bound on $x$  (or equivalently lower bound for $n-x$) for a 
given $n,k\in \N$ such that $\hat F_n(x,k)>0.$ For example, if $n=5$ and $k=1$, then  the maximum value $x$ can  assume  is $3.$

In this section, we find those $x,\;k$ for which \^{F}$_{n}(x,k)>0$ and 
then will derive a recurrence formula for \^{F}$_{n}(x,k)$. But before that, we 
explore the results for \^{F}$_{n}(x,k)$, which follow immediately from the 
definition of \^{F}$_{n}(x,k)$.
\begin{obs}\label{newww}
\begin{enumerate} 
\item\label{n even}
Suppose $n \in \mathbb{N}$ is even and $x,k$ are positive integers such that 
\^{F}$_{n}(x,k)>0$, then $x$ is even. But for even $x$ with 
\^{F}$_{n}(x,k)>0$, it is not necessary that $n$ has to be even.
\begin{proof}
It follows from the fact that a palindromic binary string of an even length 
$2m$ can be constructed from a binary string of length $m$  by concatenating to left with its mirror image.
With  $x=10$ and   $k=4$,  we have  $0100001000010\in$ \^{B}$_{13}^{10,4}$ and  
$01000011000010 \in$ \^{B}$_{14}^{10,4}.$
\end{proof}
\item For each  $n \in \mathbb{N}$, \^{F}$_{n}(0,0)=1$ and \^{F}$_{n}(n,n)=1$.\\
\item
For  $n \geq 2$, \^{F}$_{n}(n-1,k)=0$, for all $1 \leq k \leq n-1$ except 
\^F$_{n}(n-1,\frac{n-1}{2})=1$ whenever $n$ is odd.
\item \label{k half}
Suppose $n\ge 3.$ Then for each $x$, $1 
\leq x \leq n-2$ and corresponding  to each $k$, $\lfloor \frac{x}{2} \rfloor < 
k 
\leq x$, 
$$\hat {F}_{n}(x,k) =
 \begin{cases}
                \binom{\frac{n-k-2}{2}}{\frac{x-k}{2}}  & \mbox{when $x,k$ both 
has same parity as $n$,}\\
                 0 & \mbox{otherwise.}
               \end{cases}$$
\begin{proof}
Here $x=k+r$ and any X $\in$ \^B$_{n}(x,k)$ will be of the form \\
\\
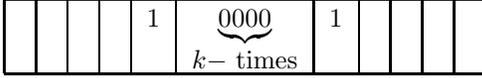
\begin{tabular}{|l|l|l|l|l|l|l|l|l|l|l|l|l|l|l|l|l
|l}
\hline
& & & & 1&$\underbrace{0000}_{\mbox{$k-$ 
times}}$&1 & & & &\\
\hline
\end{tabular}\\
\ \\
which implies that $n, x$ has the same parity as of $k$ and conversely. And the number of 
such palindromic binary strings is given by the number of binary strings of 
length $\frac{n-k-2}{2}$ containing $\frac{x-k}{2}$ zeros.\\
Hence, the result follows.
\end{proof}
\end{enumerate}
\end{obs}
From above observations, it is sufficient to find $\hat F_{n}(x,k)$ for  $n \geq 
3$,  $x\in \{2,3,\ldots,  n-2\}$ and 
$k\in \{1,2, \ldots, \lfloor \frac{x}{2} \rfloor\}$ further $x$ is even whenever 
$n$ is even. 
Before stating next result, recall $B_{n}^{x,k}$ denotes the 
set of all binary strings of length $n$ with $x$ zeros and having at least 
one  longest subword of zeros of length $k.$  

\begin{lemma}\label{lem:boundonx1}
Suppose $n,\;x,\;k \in \mathbb{N}$ and $q=\lfloor\frac{x}{k}\rfloor$. Then 
$$\hat F_{n}(x,k) > 
0\Leftrightarrow \begin{cases}
                x+q-1 \leq n  & \mbox{when \;\; $k|x,$ }\\
                x+q  \leq n  & \mbox{when \;\; $k \nmid x.$ }
               \end{cases}$$
\end{lemma}
\begin{proof}
Suppose that \^F$_n(x,k)>0$ or equivalently $X \in$  \^B$_{n}^{x,k}$. Then, 
by definition, $X$ contains $x$ 0's and $n-x$ 1's, further $X$ does not 
contain any subword of  0's of length $k+1$ or more. Hence, the maximum number of 
subwords of zeros of length $k$ in $X$ is $q$, that is the minimum number of 1's is 
$q$ or $q-1$ depending on $k\nmid x$ or $k|x$ respectively.  Equivalently, 
$ q \leq n-x $ if $k\nmid x$, and $ q 
-1 \leq n-x$   if $k | x.$

Conversely, suppose that  $q \leq n-x$ if $k\nmid x$, 
and $q-1 \leq n-x$   if $k | x.$ We have to show that $\hat B_n^{x,k}\ne 
\emptyset$ 
 in each case.\\
If  $x$ is even, then  \^B$_{n}^{x,k} \ne \emptyset$ as 
$B_{\lfloor \frac{t}{2}\rfloor}^{\frac{x}{2},k} \ne \emptyset.$
In particular, depending $q$ is even or odd and  $k | x $ or $k \nmid x $
we have 
\;\; $y\underbrace{111\cdots 1}_{\text {$n-t$ times}}y'$ or 
 $y\underbrace{111\cdots 1}_{\text {$n-t+1$ times}}y'\in \hat B_n^{x,k},$ 
where  $t= x+q$ or $x+q-1$ according as $k \nmid 
x$ or $k|x$ respectively and $y \in B_{\lfloor 
\frac{t}{2}\rfloor}^{\frac{x}{2},k}$, $y'$ is the mirror image of $y$.

If $x$ is odd, then we have the following cases.
\begin{description}
 \item[$k\nmid x$]
We have \^B$_{n}^{x,k} \ne \emptyset$ as $B_{ 
\frac{n-r-2}{2}}^{\frac{x-r}{2},k} \ne \emptyset.$  In particular,  
$y1\underbrace{000\cdots 0}_{\text {$r$ times}}1y' \in B_n^{x,k},$
 $$\underbrace{111\cdots1}_{\text{$\frac{n-x- 
q-1}{2}$ }}\;\;\underbrace{000\cdots0}_{\text{$\frac{r}{2}$ 
}}1\underbrace{YY\cdots YY}_{\mbox{$\frac{q-1}{2}$ 
}}\underbrace{000\cdots 0}_{\text{$k$ }}\underbrace{Y'Y'\cdots 
Y'Y'}_{\mbox{$\frac{q-1}{2}$ 
}}1\underbrace{000\cdots 0}_{\text{$\frac{r}{2}$ }}\;\; 
\underbrace{111\cdots 1}_{\text{$\frac{n-x- 
q-1}{2}$ }}\in B_n^{x,k}$$ depending on $r$ is odd and even, respectively.
\item[$k|x$] In this case, 
$$\underbrace{111\cdots 1}_{\text{$\frac{n-x- 
q+1}{2}$ times}}\;\;\underbrace{YY\cdots YY}_{\mbox{$\frac{q-1}{2}$ 
times}}\underbrace{000\cdots0}_{\text{$k$ times}}\underbrace{Y'Y'\cdots 
Y'Y'}_{\mbox{$\frac{q-1}{2}$ 
times}} \underbrace{111\cdots 1}_{\text{$\frac{n-x- 
q+1}{2}$ times}}\in B_n^{x,k},$$
\end{description}
where $Y = \underbrace{000\cdots 0}_{\text{k-zeros}}1$, is  a subword
of $X \in B_{n}^{x,k}$ of length $k+1$ containing $k$ consecutive 0's followed 
by a  $1$   and  $Y'$ is the mirror image of $Y$.
\end{proof}
In the above lemma, we found an upper bound for  $x$  whenever $n$ and $k$ are known. 
Since \^F$_n(x,x)>0$ for all  $0\le x\le n$ and $2|x+n$, hence,  for a given $n$ and $x$ the 
maximum value for $k$ such that \^F$_n(x,k)>0$ and $2|x+n$ is $x.$ Whereas  when $n$ is odd and $x$ is even,
 the maximum value of $k$ is 
$\frac{x}{2}$. 
For  a fixed  $n$ and $x$, 
the following result will provide the least  value of $k$ such that 
\^F$_n(x,k)>0$. 
\begin{cor}\label{Least k}
 Let $n,x\in \N$ be fixed and \^F$_n(x,k)>0$. Then the smallest  value of  $k$ 
is
$\lfloor 
\frac{n}{n-x+1} \rfloor.$
\end{cor}
\begin{theorem}\label{thm: |S_pn|}
Let $n \geq 2$ and \^S$_{n}$=\{(x,k):\^F$_{n}(x,k)>0\}$. Then when $n$ is even,
$$|\hat S_{n}|= 
\begin{cases}
1+\frac{n(n+2)}{8}-\sum\limits_{i=1}^{\frac{n}{2}}\left\lfloor \frac{n}{2i+1}\right\rfloor 
+\frac{n^2}{16}   & \mbox{when\;$4\mid n,$}\\
1+\frac{n(n+2)}{8}-\sum\limits_{i=1}^{\frac{n}{2}}\left\lfloor \frac{n}{2i+1}\right\rfloor 
+\frac{n^{2}-4}{16}   & \mbox{when\;$4 \nmid n.$}\\
\end{cases}$$
When $n$ is odd, 
$$|\hat S_{n}|= 
\begin{cases}
1+\frac{5(n+3)(n-1)}{16}-\sum\limits_{i=1}^{n-2}\left\lfloor \frac{n}{i+1}\right\rfloor    
& \mbox{when\;$4 \mid (n-1),$}\\
1+\frac{(n+3)(n-1)}{4}-\sum\limits_{i=1}^{n-2}\left\lfloor \frac{n}{i+1}\right\rfloor 
+\frac{(n+1)^{2}}{16}   & \mbox{when\;$4 \mid (n+1).$}\\
\end{cases}$$
\end{theorem}
\begin{proof}
Assume that $n$ is even. Then from part ~\ref{n even} of the Observation \ref{newww}, $x$ is even. From 
Corollary~\ref{Least k} and from part ~\ref{k half} of the Observation \ref{newww}, if  $4 | n,$ then the number of pairs $(x,k)$ such that 
\^F$_{n}(x,k) >0$ is given by

$$1+\sum\limits_{i=0}^{\frac{n}{2}-1}(i+1)-\sum\limits_{i=0}^{\frac{n}{2}-1}
\left\lfloor \frac{n}{n-2i+1}\right\rfloor+\{2(1+2+\ldots+\frac{n}{4}-1)+\frac{n}{4}\}.$$
Equivalently, if  $4 \nmid n,$ then,
$$1+\sum\limits_{i=0}^{\frac{n}{2}-1}(i+1)-\sum\limits_{i=0}^{\frac{n}{2}-1}
\left\lfloor \frac{n}{n-2i+1}\right\rfloor+\{2(1+2+\ldots+\frac{n-2}{4})\}.$$
Hence, the result follows whenever $n$ is even.

If  $n$ is odd, then $x$ can be even or odd. If $x$ is even, then from 
 Corollary~\ref{Least k} and from part ~\ref{k half} of the Observation \ref{newww}, the number of pairs $(x,k)$ such that 
\^F$_{n}(x,k) >0$ is given by
\begin{equation}
1+\sum\limits_{i=0}^{\frac{n-1}{2}}(i+1)-\sum\limits_{i=0}^{\frac{n-1}{2}}
\left\lfloor \frac{n}{n-2i+1}\right\rfloor.
\end{equation}
Suppose that $x$ is odd and $N_n(x,k)$ denote the then number of pairs $(x,k)$ such that \^F$_{n}(x,k) >0.$ If $4 | n-1,$ then $N_n(x,k)$ is given by  
is 
given by
\begin{equation}
N_n(x,k)=1+\sum\limits_{i=1}^{\frac{n-3}{2}}(i+1)-\sum\limits_{i=1}^{\frac{n-3}{2}}
\left\lfloor \frac{n}{n-2i}\right\rfloor+\{2(1+2+\ldots+\frac{n-1}{4})\}.
\end{equation}
If $4 \mid (n+1),$ then 
\begin{equation}
N_n(x,k)=1+\sum\limits_{i=1}^{\frac{n-3}{2}}(i+1)-\sum\limits_{i=1}^{\frac{n-3}{2}}
\left\lfloor \frac{n}{n-2i}\right\rfloor+\{2(1+2+\ldots+\frac{n-3}{4})+\frac{n+1}{4}\}.
\end{equation} 
Hence, the result for the case $4 | (n-1)$ follows by   adding Equation (4) and Equation (5). By adding Equation (4) and Equation (6) we get the result for the case $4 | (n+1)$.\\  
\end{proof}
\subsection{Formula for \^F$_n(x,k)$}\label{sec:PFormula}
In this section, we find a recurrence formula for $\hat F_n(x,k)$ whenever 
$2 \leq x \leq n-2$ and corresponding to each $x$, $\left\lfloor \frac{n}{n-x+1} 
\right\rfloor \leq k \leq \lfloor \frac{x}{2} \rfloor.$
\begin{theorem}\label{thm:main:recurrence1}
Let $n \geq 4$,  $x\in \{y | 2 \leq y \leq n-2, 2|n+y\}$ and 
$k\in  \{\lfloor \frac{n}{n-x+1}\rfloor,\dots, \lfloor \frac{x}{2} \rfloor\}.$ Then 
$$\hat F_{n}(x,k) = 
 \begin{cases}  
 \sum\limits_{i=0}^{\lfloor \frac{k-2}{2} \rfloor} F_{\lfloor 
\frac{n}{2}\rfloor 
-i-1}(\lfloor \frac{x}{2}\rfloor-i,k)+\sum\limits_{j=0}^{k} 
F_{\frac{n-k}{2}-1}(\frac{x-k}{2},j) & \mbox{when $2|k+n$,}\\
\sum\limits_{i=0}^{\lfloor \frac{k-1}{2}\rfloor} F_{\lfloor 
\frac{n}{2}\rfloor-i-1}(\lfloor\frac{x}{2}\rfloor-i,k) & \mbox{when $2\nmid k+n$.}
               \end{cases}$$
\end{theorem}
\begin{proof}
Suppose $n=2m$, then all palindromic binary strings of length $n$ are given by 
the concatenation of a binary string of length $m$ and its mirror image. 
For $n=2m+1$, palindromic strings are constructed by concatenation 
of a binary string of length $m$ with its mirror image keeping a central entry either $0$ or $1$.
That is 
\begin{description}
\item[$n$ is even] 
\begin{tabular}{|l|l|l||l|l|l|l|l|}
\hline
1 & 0 & 0 & 0  & 0 & 1\\
\hline
\end{tabular}
\item[$n$ is odd] 
\begin{tabular}{|l|l|l||l||l|l|l|l|l|}
\hline
1 & 0 & 0 & 0/1 &0  & 0 & 1\\
\hline
\end{tabular}
\end{description}
Thus, all palindromic binary strings are constructed 
from the elements from $B_{\lfloor \frac{n}{2} \rfloor}(\lfloor\frac{x}{2} 
\rfloor,j)$, with a condition that the length of  resulting central subword consisting 
entirely of $0$'s does not exceed $k$. If it is $k$ then 
we will choose elements from $B_{\lfloor \frac{n-k+2}{2} 
\rfloor}(\lfloor\frac{x-k}{2} \rfloor,j)$, for $0 \leq j\leq k$. We now examine 
 the construction more closely  for $n$ is even and odd, respectively.
\begin{description}
\item[n is even] Suppose $n=2m$, then $x$ is even. Suppose 
$Y^{i}=\underbrace{00\cdots 0}_{i\;times}1$, a binary string of length $i+1$. 
For 
$k$ is even, elements of   \^B$_{n}^{x,k}$ can be constructed from 
$B_{m}(\frac{x}{2},k)$ which contains $Y^{i}$ as a substring at the beginning  for  
$0 \leq i \leq  \frac{k-2}{2} $ and from $B_{m}(\frac{x}{2},j)$, for $0 \leq j 
\leq k$,  which contains $Y^{\lfloor \frac{k}{2}\rfloor}$ as a initial subword.  When $k$ is 
odd then elements will be taken from 
$B_{m}(\frac{x}{2},k)$ which consist $Y^{i}$ as a substring at starting for 
each 
$0 \leq i \leq  \frac{k-1}{2} $. Thus if $n$ is even, then
$$\hat F_{n}(x,k) = 
 \begin{cases}  
 \sum\limits_{i=0}^{\frac{k-2}{2}} 
F_{\frac{n}{2}-i-1}(\frac{x}{2}-i,k)+\sum\limits_{j=0}^{k} 
F_{\frac{n-k}{2}-1}(\frac{x-k}{2},j) & \mbox{when $k$ is even,}\\
\sum\limits_{i=0}^{\frac{k-1}{2}} F_{\frac{n}{2}-i-1}(\frac{x}{2}-i,k) & 
\mbox{when $k$ is odd.}
              \end{cases}$$
\item[n is odd] Suppose $n=2m+1$ and $x$ is odd, then the central entry is $0$.  
Suppose $k$ is even, then the elements in $\hat B_{n}^{x,k}$ will be taken 
from $B_{m}(\frac{x-1}{2},k)$ which contains $Y^{i}$ as a subword at starting 
for each $0 \leq i \leq  \frac{k-2}{2} $.\\
When $k$ is odd, then the elements in $\hat B_{n}^{x,k}$ will be taken from 
$B_{m}(\frac{x-1}{2},k)$ which contains $Y^{i}$ as a substring at starting for 
each $0 \leq i \leq  \frac{k-3}{2} $ and from 
$B_{\frac{n-k-2}{2}}(\frac{x-k}{2},j)$, for $0 \leq j \leq k$. Hence,\\
$$\hat F_{n}(x,k) =
\begin{cases}
         
\sum\limits_{i=0}^{\frac{k-2}{2}}F_{\frac{n-1}{2}-i-1}(\frac{x-1}{2}-i,k)       
 
 & \mbox{when $k$ is even,}\\
\sum\limits_{i=0}^{\frac{k-3}{2}} 
F_{\frac{n-1}{2}-i-1}(\frac{x-1}{2}-i,k)+\sum\limits_{j=0}^{k} 
F_{\frac{n-k}{2}-1}(\frac{x-k}{2},j) & \mbox{when $k$ is odd.}
               \end{cases}$$
\end{description}
Hence, the result.
\end{proof}
\begin{note}
For $n$ is odd and $x$ is even,
$\hat F_{n}(x,k) = F_{\frac{n-1}{2}}(\frac{x}{2},k).$
\end{note}
\begin{proof}
The result follows from the fact that central entry never is $0$.
\end{proof}
Now we will give a direct proof for Equation~(\ref{eq:an=Fnxk}). \\
\begin{enumerate}
\item 
Let $f_n=1+\sum\limits_{x=1}^n \hat F_{2n-1}(x,1).$  Then $f_1=1+\hat 
F_1(1,1)=2$ and 
$f_2=1+\hat F_3(1,1)+\hat F_2(2,1)=3.$ We will prove result for $n>2$ using 
Theorem~\ref{thm:main:recurrence1}. 
\begin{eqnarray*}
 f_n  &=& 2+\sum\limits_{i=1}^{\lfloor\frac{n}{2}\rfloor} 
\hat F_{2n-1}(2i,1)+\sum\limits_{i=1}^{\lfloor\frac{n-1}{2}\rfloor} 
\hat F_{2n-1}(2i+1,1)\\
&=& 2+\sum\limits_{i=1}^{\lfloor\frac{n}{2}\rfloor} 
 F_{n-1}(i,1)+\sum\limits_{i=1}^{\lfloor\frac{n-1}{2}\rfloor} 
F_{n-2}(i,1)= a_{n-1}+a_{n-2}=a_{n},\; \mbox{where}\; a_{n}=\sum\limits_{i=1}^{\lfloor\frac{n+1}{2}\rfloor} 
 F_{n}(i,1).
\end{eqnarray*}
\item
For $n \geq 2$ let $f'_{n-1}=1+\sum\limits_{i=1}^{\lfloor\frac{2n+1}{4}\rfloor} 
\hat F_{2n}(2i,1).$  Then $f'_1=1+\hat F_4(2,1)=2$ and 
$f'_2=1+\hat F_6(2,1)=3.$ We will prove result for $n>3$ using 
Theorem~\ref{thm:main:recurrence1}. 
\begin{eqnarray*}
 f'_{n-1}  &=& 1+\sum\limits_{i=1}^{\lfloor\frac{2n+1}{4}\rfloor} 
\hat F_{2n}(2i,1)= 1+\sum\limits_{i=1}^{\lfloor\frac{n}{2}\rfloor} 
 F_{n-1}(i,1)= a_{n-1}.
\end{eqnarray*}
\end{enumerate}
We now express  few of the above identities in terms of matrix equations by 
constructing a matrix using the numbers $\hat F_n(x,k).$ 
For a given $n$, $\hat F_n$ denote an $(n+1) \times (n+1)$ matrix and is 
defined 
as 
$
\hat F_n=[a_{xk}]$,
where $a_{xk}=\hat F_n(x,k)$ for $0\le x,k\le n.$
For example, $$\hat F_1=\begin{bmatrix}
                   1 & 0\\
                   0 &1
                  \end{bmatrix}, \hat F_2=\begin{bmatrix}
                  1& 0& 0,\\
                  0& 0& 0\\
                  0& 0& 1
                  \end{bmatrix}, \hat F_3=\begin{bmatrix}
                  1 & 0& 0& 0\\
                  0& 1& 0& 0\\
                 0& 1& 0& 0\\
       0& 0& 0& 1
                  \end{bmatrix}, \hat F_4=\begin{bmatrix}
                  1& 0& 0& 0& 0\\
         0& 0& 0& 0& 0\\
        0& 1& 1& 0& 0\\
        0& 0& 0& 0& 0\\
        0& 0& 0& 0& 1\\
        \end{bmatrix}.$$
Let $e_{n}$ denote the column vector in $\R^n$ with each entry 1.  Then the following hold.
\begin{enumerate}
 \item From Equation (\ref{eq:Bn=Bnxk1}), we have $e_{n+1}^T \hat F_n 
e_{n+1}=2^{\lfloor \frac{n+1}{2}\rfloor}.$
 \item From Equation (\ref{eq:ncx1}), we have 
 $\hat F_{2n} e_{2n+1}=\left[\binom{n}{0}, 0, \binom{n}{1}, 0, \ldots, 0, 
\binom{n}{n}\right]^T.$
The non-zero entries in right hand side forms the $n^{th}$ row of the Pascal triangle. 
\item Also from Equation (\ref{eq:ncx1}), we have 
 $\hat F_{2n+1} e_{2n+2}=\left[\binom{n}{0}, \binom{n}{0}, \binom{n}{1}, 
\binom{n}{1}, \ldots, 
\binom{n}{n}, \binom{n}{n}\right]^T.$
 \item The row vector $e_{n+1}^T \hat F_n$ provides column sums of $\hat F_n.$ 
 Since the first column of $\hat F_n$ is the vector $[1,0,\ldots,0]^T$ hence, the sequence
 of first column sums of $\hat F_n$ for $n=1,2,\dots$ is  
 $((e_{2n+2}^T\hat F_{2n+1})_1)_{n=1}^\infty=(1,1,1,1,\ldots).$ From 
Equation (\ref{eq:an=Fnxk}), we have 
 $$((e_{n+1}^T\hat 
F_n)_2)_{n=1}^\infty=(1,1,2,2,4,4,7,7,12,12,20,20,33,33,54,54,88,88,\ldots).$$
   \item From Theorem~\ref{thm: |S_pn|}, the number of nonzero entries in the 
matrix $\hat F_n$ is equal to
 $|\hat S_n|.$  
 \end{enumerate}
 The eigenvalues of $\hat F_n$ are $1$ and $0$  only.  The trace and 
determinant of $\hat F_n$ 
 are respectively  $1+\lfloor \frac{n+1}{2} \rfloor$ and 0.
  The matrix $\hat F_n$ in general, is not diagonalizable. For example, $\hat 
F_{5}$ is not diagonalizable whereas $\hat F_{4}$ is diagonalizable. \\
 
\subsection{Applications}\label{sec:PApplications}
\subsubsection{Number of compositions and partitions  of $n+1$}
A composition of a positive integer $n$ is a sequence $k_1,k_2,\ldots, k_r$ of 
positive 
integers called 
summands of the composition such that $n=k_1+k_2+\dots+k_r.$ A palindromic 
composition(or palindrome) is one for which the sequence reads same forwards 
and 
backwards. It is well known 
(see~\cite{H-C-G}) that 
the number of palindromes of $n$ is $2^{\lfloor \frac{n+1}{2}\rfloor}$. There 
is 
a one to one 
correspondence between 
compositions of $n$ and binary sequences of length $n-1$ (refer 
Section~5 in ~\cite{GR2}).

Few interesting applications of column sums of $\hat F_n$ are:

\begin{enumerate}
 \item the $k^{th}$ column 
sum of $\hat F_n$ is the 
 number of palindromic compositions  of $n+1$  with maximum summand $k+1$. 
Consequently, the total number of palindromic compositions of $n+1$ is 
$e_{n+1}^T \hat F_n e_{n+1}=2^{\lfloor \frac{n+1}{2}\rfloor}.$ 
\item  Further, from Equation~ (\ref{eq:an=Fnxk}),  the number of palindromic 
compositions of $2n$  with ones and twos is 
 $(e_{2n}^T \hat F_{2n-1})_1+(e_{2n}^T \hat F_{2n-1})_2=1+(e_{2n}^T \hat 
F_{2n-1})_2=a_{n}.$ And the number of palindromic compositions of $2n+1$ 
with ones and twos is 
 $(e_{2n+1}^T \hat F_{2n})_1+(e_{2n+1}^T \hat F_{2n})_2=1+(e_{2n+1}^T \hat 
F_{2n})_2=a_{n-1}.$ The same is 
derived by ~Krishnaswami Alladi and V.~E.~Hoggatt, Jr in ~\cite{K-H}.
\item
Also the total number of 1's, 2's, positive signs in all palindromic compositions 
of 
$n$ with summand 1 and 2 can be written in terms of $\hat F_{n-1}(x,k)$.\\
First, we will find out the total number of 2's. It is given by 
$\sum\limits_{x=1}^{n-1}x\;\hat F_{n-1}(x,1)$. We will denote it as 
$\mathcal{A}_{2}(n)$.
\begin{eqnarray*}
\mathcal{A}_{2}(2n+1)&=& \sum\limits_{x=1}^{n}x \; \hat F_{2n}(x,1)=2\;\sum\limits_{i=1}^{\lfloor\frac{n}{2}\rfloor} 
i\;\hat 
F_{2n}(2i,1)=2\;f_{2}(n)\\
                     &=& \mathcal{A}_{2}(2n-1)+\mathcal{A}_{2}(2n-3)+2\;\pi(2n-3),
\end{eqnarray*} 
where $f_{2}(n)$ and $\pi(n)$ have been described in ~\cite{K-H}.
\begin{eqnarray*}
\mathcal{A}_{2}(2n)&=& \sum\limits_{x=1}^{n}x \; \hat F_{2n-1}(x,1)=1+2\;\sum\limits_{i=1}^{\lfloor\frac{n}{2}\rfloor} 
i\;\hat F_{n-1}(i,1)+2\;\sum\limits_{i=1}^{\lfloor\frac{n}{2}\rfloor} i\;\hat 
F_{n-2}(i,1)+2\;\sum\limits_{i=1}^{\lfloor\frac{n}{2}\rfloor} \;\hat 
F_{n-2}(i,1)\\
                     &=&2\;f_{2}(n)+2\;f_{2}(n-1)+a_{n-2}= 
\mathcal{A}_{2}(2n-2)+\mathcal{A}_{2}(2n-4)+2\;\pi(2n-4).
\end{eqnarray*} 
The same is 
derived by ~Krishnaswami Alladi and V.~E.~Hoggatt, Jr in ~\cite{K-H}.\\
Similarly the total number of 1's and positive signs in all palindromic 
compositions 
of $n$ with summand 1 and 2 can easily be expressed in terms of $\hat 
F_{n-1}(x,1).$
\end{enumerate}
 S.~Heubach et al  ~\cite{H-C-G} counted the number of ``+" signs and number 
summands in all palindromic
 compositions  of $n$. These results can be  expressed in terms of $\hat 
F_n(x,k)$ in the 
following manner.
 \begin{enumerate}
  \item It is clear that the number of ``+" signs in a 
composition of $n$ is the same as the number of $1$'s in the corresponding binary 
string of length $n-1$. Hence, the total number of ``+" signs in all palindromic 
compositions 
of $n$  is  
$$\sum\limits_{k=0}^{x}\sum\limits_{x=0}^{n-1} x \; 
\hat F_{n-1}(x,k)=\sum\limits_{x=1}^{n-1}x \; \binom{n-1}{x}=(n-1)2^{n-2}.$$\\
If $n$ is odd, then
$$\sum\limits_{k=0}^{x}\sum\limits_{x=0}^{n-1} x \; 
\hat F_{n-1}(x,k)=2\;\sum\limits_{x=1}^{\lfloor \frac{n-1}{2}\rfloor} x \; 
\binom{\frac{n-1}{2}}{x}=\frac{n-1}{2}\;2^{\frac{n-1}{2}}.$$
If $n$ is even, then
$$\sum\limits_{k=0}^{x}\sum\limits_{x=0}^{n-1} x \; 
\hat F_{n-1}(x,k)=4\;\sum\limits_{x=1}^{\lfloor \frac{n-2}{2}\rfloor} x \; 
\binom{\frac{n-2}{2}}{x}\;+\;\sum\limits_{x=0}^{\lfloor \frac{n-2}{2}\rfloor}  
\; \binom{\frac{n-2}{2}}{x}=(n-1)\;2^{\frac{n}{2}-1}.$$
\item The number of summands in a composition is one more than the number of 
``+'' 
signs, hence, the number of summands in all compositions of $n$ is given by 
$$\begin{cases}
(n-1)2^{\frac{n}{2}-1}+\sum\limits_{k=1}^{n-1}\sum\limits_{x=1}^{n-1}\hat 
F_{n-1}(x,k)+\hat F_{n-1
}(0,0)=(n+1)2^{\frac{n}{2}-1} & \mbox{when $n$ is even,}\\
\frac{n-1}{2}2^{\frac{n-1}{2}}+\sum\limits_{k=1}^{n-1}\sum\limits_{x=1}^{n-1}
\hat F_{n-1}(x,k)+\hat F_{n-1
}(0,0)=\frac{n+1}{2}2^{\frac{n-1}{2}} & \mbox{when $n$ is odd.}
               \end{cases}$$
 \end{enumerate}
Let $\hat P_n$ denote the number of palindromic partitions of positive integer 
$n.$ Then it is easy to see that 
\begin{equation}\label{1eq: P_{n+1}{x,k}}
\hat P_{n+1}=\sum\limits_{(x,k)\in \hat S_{n}}\hat P_{n+1}(x,k),
\end{equation}
where $\hat P_{n+1}(x,k)$ denotes the number of those partitions of $n+1$ whose 
corresponding binary strings belongs to  $\hat B_{n}^{x,k}$.
 From Equation~ (\ref{1eq: P_{n+1}{x,k}}), we have $|\hat S_n|\le \hat 
P_{n+1}.$ Further, it is easy 
 to check that $|
 \hat S_n|=\hat P_{n+1}$ holds only when $n\le 5.$ 
 Hence, in order to find $\hat P_{n+1}$  it is sufficient to find $\hat 
P_{n+1}(x,k)$ for all $x,k$ with 
 $0\leq k \leq x \leq n.$ 
Clearly $\hat P_{n+1}(x,k)>0 \Leftrightarrow \hat B_{n}^{x,k} \neq 
\emptyset.$

The Equation~(\ref{1eq: P_{n+1}{x,k}}) can be derived by defining 
a  relation on $\hat B_{n}^{x,k}$ as $X\sim 
Y$ if  the lengths  of subwords of zeros are the
same in both $X$ and $Y$. Clearly, this is an equivalence relation and  $\hat 
P_{n+1}(x,k)$ is equal to 
number of equivalence classes of this equivalence relation.
Equivalently, two elements $X,Y \in \hat B_{n}^{x,k}$ correspond  to distinct 
partitions of 
$n+1$ if and only if  there exists a positive integer $ l>0 $ such that the number 
of subwords of zeros of 
length 
$l$ in $X$ and $Y$ are distinct.\\
For  $X\in \hat B_n^{x,k},$ we denote $X_l$ as the multiset of lengths of all 
subwords of zeros in $X$.
Then $X\sim Y$ if and only if $X_l=Y_l$. Hence, we can define $[X]_l=X_l.$
If $[X]_l=\{l_1,l_2,\ldots,l_t\},$ then there is a bijection between $[X]$ and 
set of all palindromic
compositions of $n+1$  having  summands $l_{1}+1,l_{2}+1,\ldots,l_{t}+1$ and 
$n-\sum\limits_{i=1}^{t}(l_{i}+1)+1$ $1$'s.

\begin{ex}
\begin{enumerate}
\item
For $n=6$, $x=4$, $k=2$ we have $\hat B_{6}^{4,2}=[X]\cup[Y],$
where $X=001100$, $Y=010010$ and
$[X]=\{001100\}\;\;[Y]=\{010010\}.$
Hence, $\hat P_{7}(4,2)=2.$
\item
 For $n=15$, $x=9$,  $k=3$ we have $|\hat B_{15}(9,3)|=10$ and 
$\hat B_{15}^{9,3}=[X]\cup[Y]\cup[Z]$, 
where $X=000111000111000$, $Y=010011000110010$, $Z=010101000101010 $ and 
\begin{eqnarray}
[X]=\{000111000111000, 100011000110001 ,110001000100011\} \nonumber
\end{eqnarray}
\begin{eqnarray}
[Y]=\{010011000110010, 001011000110100, 001101000101100\nonumber
\\
 011001000100110, 101001000100101, 100101000101001 \} \nonumber
\end{eqnarray}
$[Z]=\{010101000101010 \}.$
Hence, $\hat P_{15}(9,3)=3.$
\end{enumerate}
\end{ex}
It is easy to see that the following values of $\hat P_{n+1}(x,k)$
 follows immediately from its definition whenever $\hat F_{n}(x,k) > 0$. 
 \begin{enumerate}
\item  $\hat P_{n+1}(x,x)= \hat P_{n+1}(x,1)=1.$ And $\hat 
P_{n+1}(n-1,\frac{n-1}{2})=1$. 
\item For an odd integer $2n+1$, if $x$ is even and $n =  \frac{x}{2} +\; 
\lfloor \frac{ \frac{x}{2}+1}{2} \rfloor
+i-1,$ for nonnegative integer $i$, then 
 $\hat P_{2n+2}(x,2) = 
 \begin{cases}
             i+1   & \mbox{if $0 \leq i < \lfloor \frac{x}{4} \rfloor -1$,}\\

\lfloor \frac{x}{4} \rfloor & \mbox{if $i \geq \lfloor \frac{x}{4} \rfloor -1 
$.}
               \end{cases}$
\begin{proof}
For $x$ is even $\hat P_{2n+2}(x,2)$ is given by all possible partitions of 
$n+1$ whose corresponding binary strings are in $B_{n}^{ \frac{x}{2},2}$. Let 
$X\in B_{n}^{ \frac{x}{2},2}$ and $n_{00}$ denotes the number of times the 
string
$00$ occurs as a subword  in $X$. Then $ n_{00} \le \lfloor \frac{x}{4} 
\rfloor$ 
and 
the minimum value of 
$n_{00}$ is either $\lfloor \frac{x}{4} \rfloor-i$ or $1$
depending on   $ i < \lfloor 
\frac{x}{4} \rfloor -1$ or $i \geq \lfloor \frac{x}{4} \rfloor -1$ respectively.
Hence, the result follows.
\end{proof}

Similarly one can see if $x$ is odd and $n= \lfloor \frac{x}{2} \rfloor+\; 
\lfloor \frac{x+1}{4} 
\rfloor+i$, for nonnegative integer $i$, then
 $\hat P_{2n+2}(x,2) = 
 \begin{cases}
             i+1   & \mbox{if $0 \leq i < \lfloor \frac{x-1}{4} \rfloor -1$,}\\

\lfloor \frac{x-1}{4} \rfloor & \mbox{if $i \geq \lfloor \frac{x-1}{4} \rfloor 
-1 
$.}
               \end{cases}$           
               
\item For an even integer $2n$, if $\frac{x}{2}$ is even and $n =  \frac{3x}{4} 
+\; i$, for nonnegative integer $i$, then \\
 $\hat P_{2n+1}(x,2) = 
 \begin{cases}
             i+1   & \mbox{if $0 \leq i < \lfloor \frac{x}{4} \rfloor -1$,}\\

\lfloor \frac{x}{4} \rfloor & \mbox{if $i \geq \lfloor \frac{x}{4} \rfloor -1 
$.}
               \end{cases}$
\begin{proof}
For $\frac{x}{2}$ is even, $\hat P_{2n+2}(x,2)$ is given by all possible 
partitions of $n$ whose corresponding binary strings are in $B_{n-1}^{ 
\frac{x}{2},2}$. Let $X\in B_{n-1}^{\frac{x}{2},2}$ and $n_{00}$ denotes the 
number of times the string
$00$ occurs as a subword  in $X$. Then $ n_{00} \le \lfloor \frac{x}{4} 
\rfloor$ 
and 
the minimum value of 
$n_{00}$ is either $\lfloor \frac{x}{4} \rfloor-i$ or $1$
depending on   $ i < \lfloor 
\frac{x}{4} \rfloor -1$ or $i \geq \lfloor \frac{x}{4} \rfloor -1$ respectively.
Hence, the result follows.
\end{proof}
Similarly, if $\frac{x}{2}$ is odd and $n =  \frac{3x-2}{4} +\; i$, for 
non-negative integer 
$i$, then $\hat P_{2n+1}(x,2) = 
 \begin{cases}
             i+1   & \mbox{if $0 \leq i < \lfloor \frac{x-2}{4} \rfloor -1$,}\\

\lfloor \frac{x-2}{4} \rfloor & \mbox{if $i \geq \lfloor \frac{x-2}{4} \rfloor 
-1 
$.}
               \end{cases}$
\begin{note}
Now the number of palindromic partitions of $n+1$ with summand 1, 2 and 3 can 
easily 
be derived by adding \; $1,\hat P_{n}(x,1), \hat P_{n}(x,2)$ \; for $1 \leq x 
\leq n.$
\end{note}                               
\item If $k|x$ and $n = x+\frac{x}{k}-1$, then $\hat P_{n+1}(x,k) =1.$
\item For an even integer $x$, we have $\hat 
P_{2n+2}(x,k)=P_{\frac{n+1}{2}}\left(\frac{x}{2},k\right),$ recall  where $P_{n}(x,k)$ 
denotes the 
number of  partitions of $n+1$ whose corresponding binary strings belongs 
to  $B_{n}^{x,k}.$
\end{enumerate}
Now for given $n$, $4 \leq x \leq (n-2)$ and $3 \leq k \leq (x-2)$, $\hat 
P_{n}(x,k)$ is given by the following formula.

\begin{theorem}\label{thm:PNXK} Let $n\in \N$. Then
\begin{enumerate}
 \item \label{thm:PNXK:1} $\hat P_{2n+2}(x,k)=\sum\limits_{i=0}^{\left\lfloor\frac{k}{2}\right\rfloor-1}P_{n-i}\left(\frac{
x-2i-1}{2},k\right)+\sum\limits_{j=0}^{k}P_{\frac{2n-k+1}{2}}\left(\frac{x-k}{2}
,j\right),$
where $x$ is odd  and the second part in the right hand side vanishes whenever $k$ is even.
\item \label{thm:PNXK:2} $ \hat P_{2n+1}(x,k)=\sum\limits_{i=0}^{\lfloor\frac{k-1}{2}\rfloor}P_{n-i}\left(\frac{
x-2i}
{2},k\right)+\sum\limits_{j=0}^{k}P_{\frac{2n-k}{2}}\left(\frac{x-k}{2},j\right)
,$
where $x$ is even and the second part in the right hand side  vanishes whenever $k$ is odd.
\end{enumerate}
\end{theorem}

\begin{proof} Proof of part~\ref{thm:PNXK:1}:
To create all palindromic partitions of $2n+2$ corresponding to binary strings in  $\hat B_{2n+1}(x,k),$ combine a middle summand of size $i+1$ (with same 
parity as $2n+2$, $0 \leq i \leq \lfloor\frac{k}{2}\rfloor-1)$ with a 
partition of $\frac{2n-i+1}{2}$ on the left and its mirror image on the right. 
The process is reversible and creates no duplicates(see Lemma~2 of ~\cite{G-S}).

Proof of part~\ref{thm:PNXK:2}: Similar to the proof of part~\ref{thm:PNXK:1}.
\end{proof}

\section{Acknowledgement}
We would like to  thank Mr. Prateek Gulati and Mr. Kaustuv Deolal for computing 
various values 
of $F_n(x,k)$ using python.

\begin{thebibliography}{10}

\scriptsize
\bibitem{K-H} ~Krishnaswami Alladi, V.~E.~Hoggatt, Jr., {\it Compositions with ones 
and twos}, The Fibonacci Quarterly, (1975) 233-239.
\bibitem {G-S} ~Phyllis Chinn, Ralph Grimaldi, and Silvia Heubach, {\it The frequency of summands of a particular size in palindromic compositions,} Ars Combinatoria, 69, (2003) 65-78.
\bibitem{GR1} Ralph Grimaldi, {\it Discrete and Combinatorial Mathematics}, Pearson publication, 5th edition, 2003.
\bibitem{GR2} R.~P.~Grimaldi, S.~Heubach, {\it Binary Strings Without Odd 
Runs of zeros}, Ars Combinatoria, 75, (2005) 241-255.
\bibitem{H-C-G} ~S. Heubach, P. Z. Chinn, and R. P. Grimaldi, {\it Rises, levels, drops and ``+'' signs in compositions: extensions of a paper by Alladi and Hoggatt}, The Fibonacci Quarterly, 41(3), (2003) 229-239.
\bibitem{Koshy} Thomas Koshy, {\it Catalan Numbers with applications}, Oxford 
university press, 2009.
\bibitem {Moni:Satya} Monimala Nej,  A. Satyanarayana Reddy, {\it Exponents of primitive symmetric companion matrices}, arXiv:1806.06838v1, 2018.
\bibitem{NY1} M.~A.~Nyblom, {\it Enumerating Binary Strings 
without $r$-Runs of Ones}, International Mathematical Forum, 7(38), (2012) 1865-1876. 
\bibitem{NY2} M.~A.~Nyblom,  {\it Counting Palindromic Binary Strings
Without $r$-Runs of Ones}, Journal of Integer Sequences, 16(8), (2013)
\href{https://cs.uwaterloo.ca/journals/JIS/VOL16/Nyblom/nyblom13.html}{Article 
13.8.7}.
\bibitem{KR} Kenneth H.~Rosen, {\it Discrete Mathematics and Its 
Applications},  William C Brown Publication, 4th edition, 1998.
\bibitem{Sills} A. V. Sills, {\it Compositions, partitions, and Fibonacci numbers}, Fibonacci Quart, 49(4), (2011) 348-354.
\bibitem{AT} Amitabha Tripathi, {\it Six Ways to Count the Number of
Integer Compositions}, Crux Mathematicorum, 39(2), (2013) 84-88.


\end{thebibliography}

\end{document}